\documentclass[12pt]{article}
\usepackage{amssymb}
\usepackage{amsmath}
\usepackage{amsfonts}
\usepackage{tikz}
\usetikzlibrary{datavisualization.formats.functions}

\newcommand{\dE}{\mathbb{E}}
\newcommand{\TXs}[1]{\overline{T}_{X,#1}}
\newcommand{\TYs}[1]{\overline{T}_{Y,#1}}
\newcommand{\muXs}[1]{\widetilde{\mu}_{X,#1}}
\newcommand{\muYs}[1]{\widetilde{\mu}_{Y,#1}}
\newcommand{\sIFR}[1]{\ensuremath{#1-}IFR}
\newcommand{\sDFR}[1]{\ensuremath{#1-}DFR}
\newcommand{\sIFRA}[1]{\ensuremath{#1-}IFRA}
\newcommand{\sNBU}[1]{\ensuremath{#1-}NBU}
\newcommand{\sNBUFR}[1]{\ensuremath{#1-}NBUFR}
\newcommand{\sNBAFR}[1]{\ensuremath{#1-}NBAFR}

\newtheorem{theorem}{Theorem}

\newtheorem{corollary}{Corollary}

\newtheorem{definition}[theorem]{Definition}

\newtheorem{lemma}[theorem]{Lemma}

\newtheorem{proposition}{Proposition}
\newtheorem{remark}{Remark}

\newenvironment{proof}[1][Proof]{\noindent\textbf{#1.} }{\ \rule{0.5em}{0.5em}}

\begin{document}

\title{Iterated failure rate monotonicity and ordering relations within Gamma and Weibull distributions}
\author{Idir ARAB and Paulo Eduardo Oliveira \\
%EndAName
CMUC, Department of Mathematics, University of Coimbra, Portugal\\
e-mail: idir@mat.uc.pt\\
paulo@mat.uc.pt}

\date{ }
\maketitle

\begin{abstract}
Stochastic ordering of distributions of random variables may be defined by the relative convexity of the tail functions. This has been extended to higher order stochastic orderings, by iteratively reassigning tail-weights. The actual verification of those stochastic orderings is not simple, as this depends on inverting distribution functions for which there may be no explicit expression. The iterative definition of distributions, of course, contributes to make that verification even harder. We have a look at the stochastic ordering, introducing a method that allows for explicit usage, applying it to the Gamma and Weibull distributions, giving a complete description of the order relations within each of those families.\\


\end{abstract}
\ \ *This work was partially supported by the Centre for Mathematics of the University of Coimbra -- UID/MAT/00324/2013, funded by the Portuguese Government through FCT/MEC and co-funded by the European Regional Development Fund through the Partnership Agreement PT2020.
\section{Introduction}
Ageing and ordering notions between random variables have long attracted the interest of a wide community. These notions raise intricate theoretical problems and have been widely used in applications in reliability, actuarial science psychology or economics (see, for example, Chandra and Roy~\cite{CR01}, Nanda, Singh, Misra and Paul~\cite{NSMP02}, Franco, Ruiz and Ruiz~\cite{FRR03}, Chechile~\cite{Che11}, Belzunce, Candel and Ruiz~\cite{BCR95,BCR98}, Colombo and Labrecciosa~\cite{CL12}, Veres-Ferrer and Pavia~\cite{VFP14}). A connection with risk function properties may be found through utility functions, which may be interpreted as distribution functions. Ageing notions are usually defined in terms of the monotonicity of the survival or of the failure rate functions, while orderings between random variables, or to be more precise, their distributions, use relationships between these type of functions. The simplest ordering notions are based on direct comparisons between the survival or the failure rate functions. More interesting ordering relations, generally known as convex orderings, compare the decrease rate of the tail functions through the relative convexity between the inverse tail functions. These convex orderings have been introduced by Hardy, Littlewood and P\'{o}lya~\cite{HarLittPol59}, and some more recent results may be found in Palmer~\cite{Pal03} or Rajba~\cite{Raj14}. This means that the actual verification of these relations for given families of distributions is, in general, not obvious if the characterization of the distribution function is not simple, as is the case, for example, of the Gamma distributions. A classical early reference on ageing and some ordering problems for random variables and also on applications to reliability is the book by Barlow and Proschan~\cite{BP75}. More recent references on ageing and ordering notions, describing a nice account of properties and relations, including convex order notions, may be found in Mashall and Olkin~\cite{Marshall} or Shaked and Shanthikumar~\cite{SS07}.

A classification of families of distributions with respect to ageing notions was studied in Deshpande, Kochar and Singh~\cite{Des86} or Deshpande, Singh, Bagai and Jain~\cite{Des90}. Some of the classifications were based on higher order stochastic dominance, defined through relations between distributions constructed by iteratively reassigning their tail-weights as measures for the tails, as described in Definitions~\ref{def:age} and \ref{DEF S-IFR} below.
These iterated relations were also studied by Averous and Meste~\cite{AM89}, giving an almost complete picture of the ageing notions classification. The main focus being in establishing an hierarchy among the ageing notions rather than being very much concerned with the calculatory aspects. Naturally, the computational side of the problem becomes increasingly difficult, as a result of iterating the distribution functions. This means that, in general, simple questions as deciding whether a given distribution satisfies the appropriate monotonicity property is not simple.

As what concerns ordering notions based in the tail-weight iterated distributions, a first classification study is found in Fagiuoli and Pellerey~\cite{FP93}. Again, the main concern is in establishing different relations between the several ordering notions, essentially with no explicit examples. The same problem, considering some new ordering notions was recently studied by Nanda, Hazra, Al-Mutairi and Ghitany~\cite{ASOK}. Once more, the main interest is in studying relations between the different orderings defined, with no examples. It is interesting to note that, although there is a vast literature on ordering (and ageing) notions, there actual verification of these relations is surprisingly difficult, even for the important and popular Gamma distribution (see, for example, Khaledi, Farsinezhadb and Kochar~\cite{KFK11} or Kochara and Xub~\cite{KX11} for recent results on ordering relations within the Gamma family of distributions).

We will look at ageing and convex ordering notions having in mind the purpose of introducing an actual computationally usable methodology to decide about the iterated failure rate monotonicity and ordering relations. The paper is organized as follows: Section~\ref{sec:def} introduces the iterated distributions, gives a closed representation and defines the ageing notions, Section~\ref{sec:mon} establishes the failure rate ageing for the Weibull and Gamma distributions, Section~\ref{sec:ord} defines the convex order relation the paper is studying and studies its characterization in terms that are computationally exploitable. Finally, on Section~\ref{sec:app} we use the previous results to give a complete classification of the order relations within the Gamma and Weibull families of distributions.

\section{Definitions and basic representations}
\label{sec:def}
Let $X$ be a nonnegative random variable with density function $f_X$, distribution function $F_X$, and tail function, or, as many authors call it, survival function, $\overline{F}_X=1-F_X$. We will be interested in ageing properties depending on iterated tail-weights for the distributions, as introduced by Averous and Meste~\cite{AM89} and initially studied by Fagiuoli and Pellerey~\cite{FP93}.
\begin{definition}
\label{def:s-iter}
For each $x\geq 0$, define
\begin{equation}
\TXs{0}(x)=f_{X}(x)\quad \mbox{and}\quad \muXs{0}=\int_0^\infty \TXs{0}(t)\,dt=1.
\end{equation}
For each $s\geq 1$, define the $s-$iterated distribution $T_{X,s}$ by their tails $\TXs{s}=1-T_{X,s}$ as follows:
\begin{equation}
\TXs{s}(x)=\frac{1}{\muXs{s-1}}\int_{x}^\infty\TXs{s-1}(t)\,dt
\quad \mbox{where}\quad \muXs{s}=\int_0^\infty \overline{T}_{X,s}(t)\,dt.
\end{equation}
Moreover, we extend the domain of definition of each $\TXs{s}$ by considering $\TXs{s}(x)=1$ for $x<0$.
\end{definition}

We will be using these iterated distributions to establish ageing properties of distributions and ageing relations between different distributions within the same family. Our main concern is to introduce and use a method that actually allows the derivation of properties for specific families of distributions. For this purpose we will be exploring a closed representation for the iterated distributions.
\begin{lemma}
\label{Simple T_s}
The tails $\TXs{s}$ may be represented as
\begin{equation}
\label{eq:simpleTs}
\overline{T}_{X,s}(x)=\frac{1}{\prod_{j=1}^{s-1}\muXs{j}}\int_x^\infty \frac{(t-x)^{s-1}}{(s-1)!}f_{X}(t)\,dt.
\end{equation}
\end{lemma}
\begin{proof}
Successively replacing each $\TXs{j}$, $j=s-1,\ldots,1$, by its integral representation and reversing the integration order, we have
\begin{eqnarray*}
\lefteqn{\TXs{s}(x)%=\frac{1}{\muXs{s-1}}\int_{x}^\infty\TXs{s-1}(t)\,dt
  =\frac{1}{\muXs{s-1}}\int_{x}^{\infty}\frac{1}{\muXs{s-2}}\int_{t}^{\infty}\TXs{s-2}(u)\,du\,dt} \\
 & &=\frac{1}{\muXs{s-1}\muXs{s-2}}\int_x^\infty \int_{x}^{u}\TXs{s-2}(u)\,dt\,du \\
 & &=\frac{1}{\muXs{s-1}\muXs{s-2}}\int_{x}^{\infty}(u-x)\TXs{s-2}(u)\,du \\
 & &=\frac{1}{\muXs{s-1}\muXs{s-2}}\int_{x}^{\infty}(u-x)\frac{1}{\muXs{s-3}}\int_{u}^{\infty}\TXs{s-3}(t)\,dt\,du \\
 & &=\frac{1}{\muXs{s-1}\muXs{s-2}\muXs{s-3}}\int_{x}^\infty \frac{(t-x)^{2}}{2}\TXs{s-3}(t)\,dt \\
 & &=\cdots =\frac{1}{\prod\limits_{j=1}^{k}\muXs{s-j}}\int_{x}^\infty \frac{(t-x)^{k-1}}{(k-1)!}\TXs{s-k}(t)\,dt.
\end{eqnarray*}
So, finally, taking $k=s$, we obtain
$$
\TXs{s}(x)=\frac{1}{\prod\limits_{j=1}^{s}\muXs{s-j}}\int_{x}^\infty \frac{(t-x)^{s-1}}{(s-1)!}f_{X}(t)\,dt.
$$
To conclude the proof, just rewrite the indexing order on the product of the $\muXs{j}$.%, we have
\end{proof}
\begin{remark}
The representation (\ref{eq:simpleTs}) for $\TXs{s}$ establishes a link between the iterated distributions and stop-loss probability metrics used in actuarial models. We refer the interested reader to Rachev and R\"uschendorf~\cite{RR90} or Boutsikas and Vaggelatou~\cite{BV02} for further details.
\end{remark}
\begin{remark}
In (\ref{eq:simpleTs}), if we choose $x=0$ and take into account that $\TXs{s}(0)=1$, it
follows that
$$
\dE X^{s-1}=(s-1)!\prod_{j=0}^{s-1}\muXs{j}.
$$
Replacing this expression in (\ref{eq:simpleTs}), another representation of $\TXs{s}$ follows:
\begin{equation}
\label{eq:simple1}
\TXs{s}(x)=\frac{1}{\dE X^{s-1}}\int_x^\infty f_{X}(t)(t-x)^{s-1}\,dt.
\end{equation}
Moreover, it also follows an explicit expression for moments of the iterated distributions:
$$
\muXs{s-1}=\frac{1}{s-1}\frac{\dE X^{s-1}}{\dE X^{s-2}}.
$$
\end{remark}

We now discuss some definitions of ageing. One of the most simple and common ageing notion is based on the failure rate function of a distribution $\frac{f_X(x)}{1-F_X(x)}=\frac{\TXs{0}(x)}{\TXs{1}(x)}$. Even before getting into comparisons between probability distributions, studied later in this paper, the monotonicity of the failure rate function is a relevant property, satisfied by many common distributions. The direct verification of this monotonicity may not be a simple task, as for many distributions the tail does not have an explicit closed representation or, at least, not a manageable one. As we have defined iterated distributions it becomes now natural to proceed likewise with respect to the failure rate functions.
\begin{definition}
\label{def:s-fail}
For each $s\geq 1$ and $x\geq 0$, define the $s-$iterated failure rate function as
\begin{equation*}
r_{X,s}(x)=\frac{\overline{T}_{X,s-1}(x)}{\int_x^\infty\overline{T}_{X,s-1}(t)\,dt}
=\frac{\overline{T}_{X,s-1}(x)}{\widetilde{\mu}_{X,s-1}\overline{T}_{X,s}(x)}.
\end{equation*}
\end{definition}
It is obvious that for $s=1$, we find the failure rate of $X$:
$$
r_{X,1}(x)=\frac{\overline{T}_{X,0}(x)}{\widetilde{\mu }_{0}\overline{T}_{X,1}(x)}=%
\frac{f_{X}(x)}{1-F_{X}(x)}=\frac{f_{X}(x)}{\overline{F}_{X}(x)}.
$$
Thus the monotonicity of the failure rate is expressed as the monotonicity of $r_{X,1}$. We may extend this monotonicity notion by considering the $s-$iterared distribution, as done in Averous and Meste~\cite{AM89} and Fagiuoli and Pellerey~\cite{FP93}, among many other authors.
\begin{definition}
\label{def:age}
For $s=1,2,\ldots$, the nonnegative random variable $X$ is said to be
\begin{enumerate}
\item
\sIFR{s} (resp. \sDFR{s}) if $r_{X,s}$ is increasing (resp. decreasing) for $x\geq 0$.
\item
\sIFRA{s} if $\frac{1}{x}\int_0^x r_{X,s}(t)\,dt$ is increasing for $x>0$.
\item
\sNBU{s} if $\TXs{s}(x+t)\leq\TXs{s}(x)\TXs{s}(t)$, for all $x,t\geq 0$.
\item
\sNBUFR{s} if $r_{X,s}(0)\leq r_{X,s}(x)$, for all $x\geq 0$.
\item
\sNBAFR{s} if $r_{X,s}(0)\leq\frac{1}{x}\int_0^x r_{X,s}(t)\,dt$, for all $x>0$.
\end{enumerate}
\end{definition}
Throughout this paper we will be interested mainly in the \sIFR{s} notion. But, as proved in \cite{FP93}, this is the stronger notion. The following lemma states the relevant part, for the purposes of the present paper, of the relations between the above notions proved by Fagiuoli and Pellerey~\cite{FP93} (see their Figure~2 for an easily readable account of the relations proved).
\begin{lemma}
Let $X$ be a nonnegative random variable. Then, for each integer $s\geq 1$, the following implications hold: $X$ is \sIFR{s} $\Rightarrow$ $X$ is \sIFRA{s} $\Rightarrow$  $X$ is \sNBU{s} $\Rightarrow$ $X$ is \sNBUFR{s} $\Rightarrow$ $X$ is \sNBAFR{s}.
\end{lemma}

\section{Iterated failure rate monotonicity}
\label{sec:mon}
The iterated failure rate property of a distribution turns out not adding much to the ageing notion. Indeed, it follows from the results in Fagiuoli and Pellerey~\cite{FP93}, that it is enough to verify that $X$ is either \sIFR{1} or the \sDFR{1}, as stated next.
\begin{lemma}
\label{IFR}
Let $X$ be a nonnegative random variable. For every integer $s\geq 1$, the following relations hold.
\begin{itemize}
  \item[a)] If $X$ is \sIFR{s}, then $X$ is \sIFR{(s+1)}.
  \item[b)] If $X$ is \sDFR{s}, then $X$ is \sDFR{(s+1)}.
\end{itemize}
\end{lemma}
\begin{proof}
This is an immediate consequence of Theorems~3.4 and 4.3 in \cite{FP93}. %Fagiuoli and Pellerey~\cite{FP93}.
\end{proof}
We will now use this property to describe the failure rate monotonicity of the Weibull and the Gamma families of distributions.
We prove here the complete result for the iterated failure rate monotonicity of the Weibull and of the Gamma distributions.
\begin{theorem}
\label{weib:sIFR}
Let $X$ be a nonnegative random variable with Weibull distribution with shape parameter $\alpha$ and scale parameter $\theta$, and $s\geq 1$ an integer. If $\alpha\geq1$ (resp., $\alpha<1$), then $X$ is \sIFR{s} (resp., \sDFR{s}).
\end{theorem}
\begin{proof}
Taking into account Lemma~\ref{IFR} it is enough to consider the case $s=1$. Using the expression for the distribution function of $X$ it follows that the quotient $r_{X,1}(x)=\frac{f_X(x)}{\overline{F}_X(x)}=\frac{\alpha}{\theta}\left(\frac{x}{\theta}\right)^{\alpha-1}$, which is increasing if $\alpha\geq1$ and decreasing otherwise.
\end{proof}

We now handle the Gamma distributions. For this family of distributions, we cannot compute explicitly the failure rate function, as the distribution function has, in general, no closed form representation. So, we need a work around to prove the monotonicity. Let us start by stating, without proof, a simple but useful characterization of monotonicity.
\begin{lemma}
\label{lem:inc}
Let $g:\mathbb{R}\longrightarrow\mathbb{R}$. The function $g$ is increasing (resp. decreasing) if and only if for every $a\in\mathbb{R}$, $g(x)-a$ changes sign at most once
when $x$ traverses from $-\infty$ to $+\infty$, and if the change occurs it is in the order ``$-,+$'' (resp.``$+,-$'').
\end{lemma}

\begin{theorem}
\label{Gamma:sIFR}
Let $X$ be a nonnegative random variable with distribution $\Gamma(\alpha,\theta)$, and $s\geq1$ an integer. If $\alpha\geq 1$, then $X$ is \sIFR{s}. If $\alpha<1$, then $X$ is \sDFR{s}.
\end{theorem}
\begin{proof}
Again, from Lemma~\ref{IFR}, it follows that is enough to prove that $X$ is either \sIFR{1} or \sDFR{1}, that is, to prove the increasingness or decreasingness of the quotient $r_{X,1}(x)=\frac{f_{X}(x)}{\overline{F}_{X}(x)}$. Since, in general, there is no explicit closed form for $\overline{F}_{X}(x)$, we will prove the monotonicity using Lemma~\ref{lem:inc}. As $f_X$ and $\overline{F}_X$ are nonnegative, it is enough to take, while applying Lemma~\ref{lem:inc}, $a>0$. So, for every given $a>0$, we shall study the sign variation of $\frac{f_{X}(x)}{\overline{F}_{X}(x)}-a$. Remark that the sign of this difference coincides, for every $x\geq 0$, with the sign of $H(x)=f_{X}(x)-a\overline{F}_{X}(x)$, so it is enough to study the sign variation of $H$. It is obvious that $H(0)=-a<0$ and, if $\alpha\geq1$ we have $\lim_{x\rightarrow+\infty}H(x)=0$. Now, differentiating, we find that
$$
H^{\prime}(x)=f^{\prime}_{X}(x)+af_{X}(x)=\frac{x^{\alpha-2}e^{-x/\theta}}{\theta^{\alpha+1}\Gamma(\alpha)}((a\theta-1)x-\theta(1-\alpha)),
$$
so the sign of $H^\prime$ is determined by the sign of the straight line $\ell(x)=(a\theta-1)x-\theta(1-\alpha)$. Obviously $\ell(0)=-\theta(1-\alpha)$. Keeping in mind that we are assuming that $\alpha\geq 1$, it follows that $\ell(0)>0$, thus the sign variation of $\ell(x)$ in $[0,+\infty)$ is ``$+$'', if $(a\theta-1)>0$, or ``$+,-$'', if $(a\theta-1)<0$. In the first case, where $\ell(x)>0$, for all $x>0$, the function $H$ is always increasing so, given its value at 0 and at infinity, the sign of $H$ is ``$-$''. In the case where the sign variation of $\ell$ is ``$+,-$'', again taking into account the behaviour of $H$ at the origin and at infinity, implies that its sign variation is ``$-,+$''. The case $\alpha<1$ is analysed analogously by taking into account that $\lim_{x\rightarrow 0}H(x)=+\infty$.
\end{proof}
The \sIFR{s}-ness of the Gaussian distributions is proved in an analogous way. We state the result without proof.
\begin{theorem}
Gaussian distributions are \sIFR{s}.
\end{theorem}

\section{Iterated failure rate ordering} %$s$-IFR ordering}
\label{sec:ord}
We now compare different distributions with respect to their iterated failure rate monotonicity rates. On the sequel, let $\mathcal{F}$ denote the family of distributions functions such that $F(0)=0$ and the corresponding probability distribution has support contained in $[0,+\infty)$. In this section, we will define an iterated failure rate order and prove a general criterium. We start by defining the ordering, following Nanda, Hazra, Al-Mutairi and Ghitany~\cite{ASOK}.
\begin{definition}
\label{DEF S-IFR}
Let $X$ and $Y$ be random variables with distribution functions $F_X,F_Y\in\mathcal{F}$ and  $s\geq 1$ an integer. The random variable $X$ (or its distribution $F_X$) is said to be more \sIFR{s} than $Y$ (or its distribution $F_Y$), and we write $X\leq_{\sIFR{s}}Y$, or equivalently, $F_X\leq_{\sIFR{s}}F_Y$, if $c_s(x)=\TYs{s}^{-1}(\TXs{s}(x))$ is convex.

\smallskip

\noindent
Moreover, two nonnegative random variables $X$ and $Y$, or two distribution functions $F_X,\,F_Y\in \mathcal{F}$, are said to be \sIFR{s} equivalent, denoted by $X\sim_{\sIFR{s}} Y$ or $F_X\sim_{\sIFR{s}} F_Y$, if there exists a constant $k>0$ such that $F_X(x)=F_Y(kx)$, for all $x\geq 0$.
\end{definition}
The \sIFR{s} relation between random variables, or their distributions to be more precise, is a variant of relative convexity between real functions $f_1$ and $f_2$, as defined by Hardy, Littlewood and P\'{o}lya~\cite{HarLittPol59}, Pe\v{c}ari\'{c}, Proschan and Tong~\cite{PecProschTong92} or Roberts and Varberg~\cite{RobVar73}. However, these authors define the relative convexity of $f_1$ with respect to $f_2$ requiring the convexity of $f_1(f_2^{-1}(x))$, that is, using the inverse functions in reversed order when compared to Definition~\ref{DEF S-IFR}. For some more recent results on the characterization of relative convexity we may refer the reader to Palmer~\cite{Pal03} or Rajba~\cite{Raj14}. These authors give some equivalent characterizations of relative convexity, but they all depend on the functions that are compared. So, if we do not have closed representations for these functions, as is the case for the distributions we will be analysing below, the effective calculation difficulty remains.

It is possible to define several other ordering relations corresponding to the different ageing notions referred in Definition~\ref{def:age}, as was done by Nanda, Hazra, Al-Mutairi and Ghitany~\cite{ASOK}. We will be only interested in the \sIFR{s} ordering, so we do not quote here those other ordering notions. Moreover, as happens for the ageing notions, the \sIFR{s} ordering is the strongest of those order relations, as proved in \cite{ASOK}. We have been referring to the \sIFR{s} as an ordering but, of course, one has to verify that this is really the case. This has been proved in~\cite{ASOK}.
\begin{lemma}[Theorem 2.1 in \cite{ASOK}]
\label{partial}
The relationship $F_X\leq_{\sIFR{s}}F_Y$ defines an order relation on the equivalence classes with respect to $\sim_{\sIFR{s}}$, of $\mathcal{F}$.
\end{lemma}

\begin{remark}
This order relation is indeed only partial, as shown by the following example. Consider $X$ with inverse gamma distribution with shape parameter $\alpha=1$ and scale parameter $\beta>0$, $Y$ with exponential distribution with scale parameter $1/\lambda$, and consider $s=1$. Then, we have:
$$
f_X(x)=\frac{\beta}{x^{2}}e^{-\beta/x},\qquad\TXs{1}(x)=\overline{F}_X(x)=1- e^{-\beta/x},\quad\mbox{and}\quad \TYs{s}(x)=e^{-\lambda x},
$$
so,
$$c_1(x)=-\frac{1}{\lambda}\log\frac{f_{X}(x)}{\overline{F}_{X}(x)}=-\frac{1}{\lambda}\log\frac{\beta e^{-\beta/x}}{x^{2}(1- e^{-\beta/x})},
$$
is neither convex nor concave, thus $X$ and $Y$ are not comparable with respect to \sIFR{1}.
\end{remark}
From Definition~\ref{DEF S-IFR} and Lemma~\ref{partial}, it follows immediately that the multiplying random variables by positive constants will not affect the \sIFR{s} ordering relation.
\begin{corollary}
\label{cor:scale}
Let $X$ and $Y$ be nonnegative random variables with distributions $F_X,F_Y\in\mathcal{F}$, $s\geq 1$ an integer and $\alpha_1,\alpha_2>0$. If $X\leq_{\sIFR{s}}Y$, then $\alpha_1 X\leq_{\sIFR{s}}\alpha_2 Y$.
\end{corollary}
The previous result will be useful to compare parametric distributions where there exists a scale parameter, as it follows that we may assume this parameter to be equal to 1.

The exponential distribution plays an important role when dealing with ageing notions. As already proved by Nanda, Hazra, Al-Mutairi and Ghitany~\cite{ASOK}, the \sIFR{s} comparability with the exponential is equivalent to the failure rate monotonicity.
\begin{theorem}[Theorem 2.2 in \cite{ASOK}]
\label{thm:expon}
Let $X$ be a random variable with distribution function $F_X\in\mathcal{F}$ and $Y$ with exponential distribution with scale parameter $1/\lambda$. Then $X\leq_{\sIFR{s}}Y$ (resp., $Y\leq_{\sIFR{s}}X$) if and only if $X$ is \sIFR{s} (resp., $X$ is \sDFR{s}).
\end{theorem}
\begin{proof}
We have $\TYs{s}=e^{-\lambda x}$, thus $F_X\leq_{\sIFR{s}}F_Y$ is equivalent to $\log(\TXs{s}(x))$ being concave, which is equivalent to requiring that $r_{X,s}(x)$ is increasing, that is, $X$ is \sIFR{s}.
\end{proof}

Note that this characterization provides an alternative way to our statements about the \sIFR{s}-ness of the Weibull and Gamma distributions (Theorems~\ref{weib:sIFR} and \ref{Gamma:sIFR} above). Of course, to use Theorem~\ref{thm:expon} we still need an effective way to compare distributions with respect to \sIFR{s} order relation, and this may not be a simple task. Indeed,
the direct verification of the convexity of $c_s$, stated in Definition~\ref{DEF S-IFR}, is in general difficult to perform, as we cannot find explicit closed representations of the distributions functions involved in the definition of $c_s$, thus we cannot invert $\TYs{s}$. One could try to use the characterization of the derivative of the inverse function for this purpose. This is exactly what was done in Proposition~2.2 in Nanda, Hazra, Al-Mutairi and Ghitany~\cite{ASOK} to obtain alternative characterizations for the \sIFR{s} ordering. But these alternatives are not really effective for actual computation purposes, as they all depend on monotonicity relations of transformations of the iterated distribution functions and their inverses. Thus, in all cases where no explicit closed representations is available, as for the Gamma family, we still have no effective way to conclude about the order relation. As already commented above, the characterizations proved by Palmer~\cite{Pal03} or Rajba~\cite{Raj14} do not help on this matter.

We shall start by proving an alternative characterization for the convexity of a continuous real function in terms of crossings of their graphical representations with straight lines.
\begin{theorem}
\label{thm:conv}
Let $f$ be a continuous function. The function $f$ is convex if and only if for every real numbers $a$ and $b$, $f(x)-(ax+b)$ changes sign at most twice
when $x$ traverses from $-\infty$ to $+\infty$, and if the change of sign occurs twice it is in the order ``$+,-,+$''.
\end{theorem}
\begin{proof}
Assume that $f(x)-(ax+b)$ changes sign more than twice or in the order ``$-,+,-$''. In both cases there exists an interval where the sign change sequence is in the order ``$-,+,-$''. But this means that the function $f$ is not convex as, after getting above a straight line it crosses again under the same straight line.

Assume now that $f$ is not convex, then there exists an interval $I=[x_0,x_1]$ such that $f(x)>\frac{f(x_1)-f(x_0)}{x_1-x_0}(x-x_0)+f(x_0)$, for all $x\in(x_0,x_1)$, that is, the graph of $f$ is, for $x\in I$, above the line $\Delta$ defined by $(x_{0},f(x_{0}))$ and $(x_{1},f(x_{1}))$. Let  $\Delta_\varepsilon$ be the line obtained by shifting upwards $\Delta$ by $\varepsilon$, described $y=\frac{f(x_1)-f(x_0)}{x_1-x_0}(x-x_0)+f(x_0)+\varepsilon$. It is obvious that, at least for $\varepsilon$ small enough, the sign variation of $f(x)-\left(\frac{f(x_1)-f(x_0)}{x_1-x_0}(x-x_0)+f(x_0)+\varepsilon\right)$, for $x\in I$, is at least in the order ``$-,+,-$''.
\end{proof}
To complement the previous result, the following characterization of the crossing of two graphical representations will be useful.
\begin{lemma}[Marshal and Olkin~\cite{Marshall}, pp. 699--700]
\label{lem:marshall}
Let $f$ and $g$ two real-valued functions, and $\zeta$ be a strictly increasing (resp., decreasing) and continuous function defined on the range of $f$ and $g$. For any real number $c>0$, the functions $f(x)-cg(x)$ and $\zeta(f(x))-\zeta(cg(x))$ have the same (resp., reverse) sign variation order as $x$ traverses from $-\infty$ to $+\infty$.
\end{lemma}

The previous results provide an immediate and simple alternative characterization of \sIFR{s} order relation.
\begin{theorem}
\label{convexity-equivalence}
Let $X$ and $Y$ be random variables with distribution functions $F_X,F_Y\in\mathcal{F}$. $X<_{\sIFR{s}}Y$ if and only if for any real numbers $a$ and $b$, $\TYs{s}(x)-\TXs{s}(ax+b)$ changes sign at most twice, and if the change of signs occurs twice, it is in the order ``$+,-,+$'', as $x$ traverses from $0$ to $+\infty$.
\end{theorem}
\begin{remark}
We have reduced the variation of $x$ to traversing from 0 to $+\infty$ because all the functions $\TXs{s}$ and $\TYs{s}$ are equal to 1 for $x<0$.
\end{remark}
\begin{definition}
Given random variables $X$ and $Y$, we denote $V_{s}(x)=\TYs{s}(x)-\TXs{s}(ax+b)$.
\end{definition}
It is obvious from the definition of the iterated tails that $V_s$ is differentiable.

\begin{remark}
\label{a<0}
Taking into account that $\TXs{s}$ and $\TYs{s}$ are decreasing, being the tails of distributions, it is enough to consider, when applying
Theorem~\ref{convexity-equivalence}, the constant $a>0$. Indeed, we have $V(0)=1-\TXs{s}(b)$, and if $a<0$,
$$
V_{s}^\prime(x)
=-\frac{1}{\muYs{s-1}}\TYs{s-1}(x)+\frac{a}{\muXs{s-1}}\TXs{s-1}(ax+b)\leq 0.
$$
Now, it is obvious that for $a<0$, we have $\lim_{x\rightarrow+\infty}V_{s}(x)= -1$. Thus, the sign variation of $V$ will be``$+,-$'', if $b>0$, and ``$-$'' if $b\leq0$. That, in both cases we meet the convexity condition described in Theorem~\ref{thm:conv}.
\end{remark}
Finally, we prove a simple result describing the sign variation after performing integration. This will be convenient for the later discussion.
\begin{lemma}
\label{sign-integral}
Let $f$ and $g$ be two real-valued functions defined on $[0,\infty)$ such that
$$
g(x)=\int_{x}^\infty f(t)\,dt.
$$
Assume that, as $x$ traverses from $0$ to $+\infty$, $f(x)$ changes sign in one of the following orders ``$-,+$'' or ``$+,-$'' or ``$+,-,+$'' or ``$-,+,-,+$''. Then $g(x)$, as $x$ traverses from $0$ to $\infty$, has sign variation equal to every possible final part of the sign variation  for $f(x)$.
\end{lemma}
\begin{proof}
The proof follows from a simple argument using that $g^\prime(x)=-f(x)$, and separating into the four possible sign variations considered. %Note that, except for the sign variation ``$+,-$'', $f$ will be, for all $x$ large enough, nonnegative.
%\begin{description}
%\item[{\rm\textit{Sign variation of $f$: ``$-,+$''}.}]
%%$f$ changes the sign in the order $"-,+"$.
%Then $g$ is first increasing and afterwards decreasing. So, if $g(0)\geq 0$ it follows that $g(x)\geq 0$, for all $x\in[0,+\infty)$. On the other hand, if $g(0)<0$, as $g$ will eventually be positive, it follows that the sign variation of $g$ is ``$-,+$''.
%
%\item[{\rm\textit{Sign variation of $f$: ``$+,-$''}.}]
%%$f$ changes the sign in the order $"+,-"$.
%This follows as the previous case (remember that in this case $f$ will be, for all $x$ large enough, negative), reversing the inequalities.
%
%\item[{\rm\textit{Sign variation of $f$: ``$+,-,+$''}.}]
%%$f$ changes the sign in the order $"+,-,+"$.
%Again, using the monotonicity regions for $g$, it follows easily that if $g(0)\leq 0$ then $g$ change the sign in the order ``$-,+$''. If $g(0)>0$, we have two possibilities, either $g(x)\geq 0$, for every $x\in[0,+\infty)$, or $g$ changes sign in the order ``$+,-,+$''.
%
%\item[{\rm\textit{Sign variation of $f$: ``$-,+,-,+$''}.}]
%%$f$ changes sign in the order $"-,+,-,+"$.
%Remembering that $g^\prime=-f$, the monotonicity of $g$ is $\nearrow\searrow\nearrow\searrow$. If $g(0)\geq 0$, the sign variation of $g$ may be ``$+$'' or, if the oscillation is large enough, ``$+,-,+$''. When $g(0)<0$, the sign variation of $g$ may be ``$-,+$'' or ``$-,+,-,+$''.
%\end{description}
\end{proof}

We may now prove a general criterium to compare, with respect to the \sIFR{s} order, two distribution functions.
\begin{theorem}
\label{main}
Let $X$ and $Y$ be random variables with absolutely continuous distributions with densities $f_X$ and $f_Y$, and distribution functions $F_X,F_Y\in\mathcal{F}$, respectively. If, for some positive integer $k\leq s$, and every $a>0$ and $b\in\mathbb{R}$, the function
\begin{equation}
 \label{eq:H}
H_k(x)=\frac{1}{\prod_{j=1}^{k}\muYs{s-j}} \TYs{s-k}(x)-\frac{a^{k}}{\prod_{j=1}^{k}\muXs{s-j}}\TXs{s-k}(ax+b)
\end{equation}
changes sign at most twice, and if the change of signs occurs twice, it is in the order ``$+,-,+$'', as $x$ traverses from $0$ to $+\infty$, then $F_X\leq_{\sIFR{s}}F_Y$.
\end{theorem}
\begin{proof}
Remember the integral representation for $V_{s}(x)=\TYs{s}(x)-\TXs{s}(ax+b)$ obtained in the intermediate steps of the proof of Lemma~\ref{Simple T_s}:
\begin{eqnarray*}
\lefteqn{V_{s}(x)=
 \frac{1}{\prod_{j=1}^{k}\muYs{s-j}}\int_{x}^\infty \frac{(t-x)^{k-1}}{(k-1)!}\TYs{s-k}(t)\,dt} \\
 & &\qquad\qquad
 -\frac{1}{\prod_{j=1}^{k}\muXs{s-j}}\int_{ax+b}^\infty \frac{(t-(ax+b))^{k-1}}{(k-1)!}\TXs{s-k}(t)\,dt \\
 & &=\int_x^\infty\frac{(t-x)^{k-1}}{(k-1)!}H_k(t)\,dt,
\end{eqnarray*}
after an appropriate change of variable in the second integral. Now, using Theorem~\ref{convexity-equivalence} and Lemma~\ref{sign-integral} the proof is concluded.
\end{proof}
\begin{remark}
As mentioned before, in general, the explicit form of $\TXs{s}$ and $\TYs{s}$ are difficult to obtain. So, in most of the cases we will be interested in applying Theorem~\ref{main} choosing $k=s$, thus using the density functions to define $H_s$, or $k=s-1$, using the distribution functions to define $H_{s-1}$, if those are available.
\end{remark}
Following the previous remark, we have a closer look to $H_s$ and $H_{s-1}$, and the control of their sign variation. Taking into account the representation (\ref{eq:simple1}) for the iterated tails, we have
$$
H_s(x)=\frac{1}{\dE Y^{s-1}}f_Y(x)-\frac{a^s}{\dE X^{s-1}}f_X(ax+b)
$$
and
$$
H_{s-1}(x)=\frac{1}{\dE Y^{s-2}}\overline{F}_Y(x)-\frac{a^{s-1}}{\dE X^{s-2}}\overline{F}_X(ax+b).
$$
In most cases, the direct analysis of the sign variation of $H_s$ is, to say the least difficult, even for relatively simple density functions as, for example, the Gamma densities. An alternative approach to the control of this sign variation is to apply Lemma~\ref{lem:marshall}, choosing an appropriate $\zeta$ transformation. For the family of distributions we will be considering in the sequel, we shall take $\zeta(x)=\log x$.
\begin{corollary}
\label{maincor}
Let $X$ and $Y$ be random variables with absolutely continuous distributions with densities $f_X$ and $f_Y$ and distribution functions $F_X,F_Y\in\mathcal{F}$, respectively. If, for every constants $a>0$ and $b\in\mathbb{R}$, either of the functions,
$$
P_s(x)=\log f_Y(x)-\log f_X(ax+b)+\log\frac{\dE X^{s-1}}{a^s\dE Y^{s-1}},
$$
or
$$
P_{s-1}(x)=\log \overline{F}_Y(x)-\log\overline{F}_X(ax+b)+\log\frac{\dE X^{s-2}}{a^{s-1}\dE Y^{s-2}},
$$
changes sign at most twice
when $x$ traverses from $0$ to $+\infty$, and if the change of sign occurs twice it is in the order ``$+,-,+$'', then $F_X\leq_{\sIFR{s}}F_Y$.
\end{corollary}

\section{Some applications}
\label{sec:app}
In this section, we will be applying the general characterizations derived before to establish the \sIFR{s} ordering among some families of distributions.

\subsection{Comparing two Gamma distributions}
As argued after Corollary~\ref{cor:scale}, it is enough to compare Gamma distributions both with the same scale parameter $\theta=1$. We will be using Corollary~\ref{maincor} with respect to $P_s$, assuming $X$ has $\Gamma(\alpha^\prime,1)$ distribution and $Y$ has $\Gamma(\alpha,1)$ distribution. Thus, we need to analyse the sign variation in $[0,+\infty)$ of
\begin{equation}\label{eq:Ps}
P_s(x)=(\alpha-1)\log x-(\alpha^\prime-1)\log (ax+b)-x+ax+b
          +\log\frac{\Gamma(\alpha^\prime)}{\Gamma(\alpha)}+\log\frac{\dE X^{s-1}}{a^{s}\dE Y^{s-1}}.
\end{equation}
where $a>0$ and $b\in\mathbb{R}$. Note that $\lim_{x\rightarrow+\infty}P_s(x)=\infty\times{\rm sgn}(a-1)$. Differentiating the expression above, we have
\begin{equation}\label{eq:Ps2}
\renewcommand{\arraystretch}{2.5}
\begin{array}{rcl}
P_s^\prime(x) & = &\displaystyle \frac{\alpha-1}{x}-\frac{a(\alpha^\prime-1)}{ax+b}+a-1 \\
 & = &\displaystyle \frac{a(a-1)x^2+((\alpha-\alpha^\prime) a+(a-1)b)x+(\alpha-1) b}{x(ax+b)}.
\end{array}
\renewcommand{\arraystretch}{1}
\end{equation}
Let us denote the numerator in the expression of $P_s^\prime$ by $N_s(x)=a(a-1)x^2+((\alpha-\alpha^\prime-) a+(a-1)b)x+(\alpha-1)b$.
To analyse the sign variation of $V_s$, we need to separate between the cases when $b\geq 0$ and $b<0$. Indeed, while for the first case we need to consider $x$ traversing from 0 to $+\infty$, for the later case, we will be only analysing the sign variation in the interval $(-\frac{b}{a},+\infty)$ as, for $x\leq-\frac{b}{a}$, $V_{s}(x)=\TYs{s}(x)-1\leq0$. Hence for both cases, in the interval of interest, the sign of $P_s^\prime$ is determined by the sign of $N_s$.
\begin{proposition}
\label{gamma_gamma1}
Let $\alpha^\prime>\alpha>1$ and $\theta_1,\theta_2>0$. The $\Gamma(\alpha^\prime,\theta_1)$ distribution is more \sIFR{s} than the $\Gamma(\alpha,\theta_2)$ distribution.
\end{proposition}
\begin{proof}
Taking into account Corollary~\ref{cor:scale}, we may assume $\theta_1=\theta_2=1$.
Moreover, remember that, according to Remark~\ref{a<0}, it is enough to take $a>0$. Note still that $\lim_{x\rightarrow 0^+}P_s(x)=-\infty$.

Assume first that $b\geq0$. The convexity of $N_s$ is determined by the sign of $a-1$, and $N_s(0)=(\alpha-1)b\geq0$, so the behaviour of $P_s^\prime$ may be as follows:
\begin{center}
\begin{tabular}{ccc}
\begin{tikzpicture}[below,scale=.5]
\draw[->] (-.5,0) -- (7.5,0) node [below] {\small $x$};
\draw[->] (0,-2.5) -- (0,4.0);
\draw (.2,2.5) to [out=280,in=135] (1.5,0) to [out=315,in=170] (2.5,-1.0) [out=5,in= 225] to (3.5,0) to [out=45,in=183] (5.75,1);
\end{tikzpicture}
 & \qquad\qquad &
\begin{tikzpicture}[below,scale=.5]
\draw[->] (-.5,0) -- (7.5,0) node [below] {\small $x$};
\draw[->] (0,-2.5) -- (0,4.0);
\draw (.2,2.5) to [out=280,in=135] (1.5,0) to [out=315,in=170] (2.5,-1.5) [out=5,in= 225] to (3.5,-1.0) to [out=45,in=183] (5.75,-.5);
\end{tikzpicture} \\
 & & \\
{\small $P_s^\prime$ for $a>1$} & & {\small $P_s^\prime$ for $a<1$}
\end{tabular}
\end{center}%
\begin{description}
    \item[{\rm\textit{Case 1} ($a>1$)}] We have $\lim_{x\rightarrow+\infty}P_s(x)=+\infty$, thus the most sign varying situation corresponds to ``$-,+,-,+$'' implying, based on\linebreak Lemma~\ref{sign-integral}, that the sign variation of $V_s$ might be ``$-,+,-,+$'' or ``$+,-,+$'' or ``$-,+$'' or ``$+$''. Now, as $V_s(0)=1-\TXs{s}(b)\geq0$, the only possible cases are ``$+,-,+$'' or ``$+$''.
    \item[{\rm\textit{Case 2} ($a\leq1$)}] In this case we have $\lim_{x\rightarrow+\infty}P_s(x)=-\infty$. The behaviour of $P_s^\prime$ when $a=1$ is still described by the picture on the right, with $P_s^\prime$ approaching 0 as $x\longrightarrow+\infty$, instead of being strictly negative. Taking into account this behaviour of $P_s^\prime$ the monotonicity of $P_s$ is, $\nearrow\searrow$, meaning that the most sign varying case for $P_s$ is ``$-,+,-$''. Again, based on Lemma~\ref{sign-integral} and $V_s(0)\geq 0$, the only possible sign variation is ``$+,-$''.
    \end{description}

\smallskip

\noindent
Assume now that $b<0$. Then, we have, for $x\leq-\frac{b}{a}$, $V_{s}(x)=\TYs{s}(x)-1\leq0$, so it remains to describe the sign variation for $x>-\frac{b}{a}$, thus needing to locate $-\frac{b}{a}$ with respect to the roots of $N_{s}$. As $N_s(0)=(\alpha-1)b<0$, two situations may occur:
    \begin{description}
    \item[{\rm\textit{Case 3} ($a>1$)}]  Then, the sign variation of $N_s(x)$ in $(-\frac{b}{a},+\infty)$ is either ``$+$'' or ``$-,+$''. As $\lim_{x\rightarrow(-b/a)^+}P_s(x)=+\infty$ and $\lim_{x\rightarrow+\infty}P_s(x)=+\infty$, it follows that the sign variation of $N_s(x)$ is ``$-,+$''. Thus, the most sign varying possibility for $P_s$ in the interval $(-\frac{b}{a},+\infty)$ is ``$+,-,+$''. From Lemma~\ref{sign-integral}, it follows that the sign variation for $V_s$ in $(-\frac{b}{a},+\infty)$ is one of the three possibilities: ``$+,-,+$'' or ``$-,+$'' or ``$+$''. As $V_s(-\frac{b}{a})\leq 0$, it follows that the sign variation of $V_s$ in $(0,+\infty)$ is ``$-,+$''.

    \item[{\rm\textit{Case 4} ($a\leq1$)}]
     In this case, the sign variation of $N_s(x)$ in the interval $(-\frac{b}{a},+\infty)$ is either ``$-$'', or ``$-,+,-$''. Assume first that the sign variation of $N_s(x)$ is ``$-,+,-$'', which means that $N_s$ has two positive roots and  its maximum is reached for $x=-\frac{b}{2a}+\frac{\alpha-\alpha^\prime}{2(1-a)}<-\frac{b}{a}$, therefore, the sign variation of $N_{s}$ in the interval $(-\frac{b}{a},+ \infty)$ is ``$+,-$'' or ``$-$''. As $\lim_{x\rightarrow(-b/a)^+}P_s(x)=+\infty$, there is only one possible sign variation of $N_{s}$, which is ``$-$''. Hence, the sign variation of $P_{s}$ is, at most, ``$+,-$''. Now, using the fact that $V_s(-\frac{b}{a})\leq 0$, it follows the sign variation of $V_s$ in $[0,+\infty)$ is ``$-$''. It remains to analyse the case where $N_s$ is always negative, but the description of the sign variation of $V_s$ follows in the same way.
    \end{description}
So, finally, the possibilities for the sign variation of $V_s$ are: either at most one sign change or, in case of two sign changes, these are ``$+,-,+$''. Hence, the conclusion follows taking into account Theorem~\ref{convexity-equivalence}.
\end{proof}
\begin{proposition}
\label{gamma_gamma2}
Let $\alpha^\prime>1>\alpha>0$ and $\theta_1,\theta_2>0$. The $\Gamma(\alpha^\prime,\theta_1)$ distribution is more \sIFR{s} than the $\Gamma(\alpha,\theta_2)$ distribution.
\end{proposition}
\begin{proof}
The result follows immediately using Theorem~\ref{thm:expon} and the transitivity of the \sIFR{s}-order, by comparing both of them with the exponential distribution.
\end{proof}

\subsection{Comparing two Weibull distributions}
On the sequel, we shall denote by $W(\alpha,\theta)$ the Weibull distribution with shape parameter $\alpha$ and scale parameter $\theta$. As for the Gamma family of distributions, it is enough to compare Weibull distributions both with scale parameter $\theta=1$. Moreover, we will apply Corollary~\ref{maincor} now with respect to $P_{s-1}$, as the tail of a Weibull distribution has a simple closed form representation, assuming that $X$ has distribution $W(\alpha^\prime,1)$ and $Y$ has distribution $W(\alpha,1)$. So, we are interested in analysing the sign variation of
\begin{equation}
\label{eq:Weib1}
P_{s-1}(x)=-x^\alpha+(ax+b)^{\alpha^\prime}+\log\frac{\dE X^{s-2}}{a^{s-1}\dE Y^{s-2}},
\end{equation}
where $a>0$ and $b\in\mathbb{R}$. Differentiating this expression, we have
\begin{equation}
\label{eq:Weib2}
P_{s-1}^\prime=a\alpha^\prime(ax+b)^{\alpha^\prime-1}-\alpha x^{\alpha-1}.
\end{equation}
The direct control of the sign variation of $P_{s-1}^\prime$ is too difficult, so we will use again Lemma~\ref{lem:marshall} with the choice $\zeta(x)=\log x$. This means that the sign variation of $P_{s-1}^\prime$ is the same as the sign variation of
%Actually $P_{s-1}^\prime$ is not a polynomial
\begin{equation}
\label{eq:Weib3}
\renewcommand{\arraystretch}{1.5}
\begin{array}{rcl}
Q_{s-1}(x)&=&\log(a\alpha^\prime(ax+b)^{\alpha^\prime-1})-\log(\alpha x^{\alpha-1}) \\
 &=&\log(a\alpha^\prime)-\log\alpha+(\alpha^\prime-1)\log(ax+b)-(\alpha-1)\log x,
\end{array}
\renewcommand{\arraystretch}{1}
\end{equation}
whose derivative is
\begin{equation}
\label{eq:Weib4}
Q_{s-1}^\prime(x)=\frac{a(\alpha^\prime-1)}{ax+b}-\frac{\alpha-1}{x}
=\frac{a(\alpha^\prime-\alpha)x+(1-\alpha)b}{x(ax+b)}.
\end{equation}

\begin{proposition}
\label{Weib_Weib1}
Let $\alpha^\prime>\alpha>1$ and $\theta_1,\theta_2>0$. The $W(\alpha^\prime,\theta_1)$ distribution is more \sIFR{s} than the $W(\alpha,\theta_2)$ distribution.
\end{proposition}
\begin{proof}
As before, without loss of generality, we may take $\theta_1=\theta_2=1$ and use the representations (\ref{eq:Weib1})--(\ref{eq:Weib4}). As usual, we need to separate the cases $b>0$ and $b\leq 0$, and remember that we need only to assume that $a>0$.

\smallskip

\noindent
Assume first that $b>0$. It follows from (\ref{eq:Weib4}) that the sign variation in the interval $(0,+\infty)$ for $Q_{s-1}^\prime$ is ``$-,+$'', hence the monotonicity of $Q_{s-1}$, in this same interval, is $\searrow\nearrow$. From (\ref{eq:Weib3}), it follows that $\lim_{x\rightarrow0^+}Q_{s-1}(x)=+\infty$ and $\lim_{x\rightarrow+\infty}Q_{s-1}(x)=+\infty$, so the most sign varying possibility for $Q_{s-1}$, which coincides with the sign variation of $P_{s-1}^\prime$, is ``$+,-,+$''. It follows that the monotonicity of $P_{s-1}$ is $\nearrow\searrow\nearrow$. Note that $\alpha^\prime>\alpha$ implies $\lim_{x\rightarrow+\infty}P_{s-1}(x)=+\infty$. Hence, the most sign varying possibility in the interval $(0,+\infty)$ for $P_{s-1}$ is ``$-,+,-,+$''. Taking now into account Lemma~\ref{sign-integral}, the sign variation of $V_s$ in the interval $(0,+\infty)$ may be ``$-,+,-,+$'' or  ``$+,-,+$'' or ``$-,+$'' or ``$+$'', and remembering that $V_{s}(0)\geq 0$, the actual possible choices are ``$+,-,+$'' or ``$+$''.

\smallskip

\noindent
Assume now that $b\leq 0$. As explained before, we need only to describe the sign variation in $(-\frac{b}{a},+\infty)$. Now from (\ref{eq:Weib4}) it follows that, for $x>0$, we have  $Q_{s-1}^\prime(x)>0$, so $Q_{s-1}$ is always increasing in $(0,+\infty)$. As $\alpha^\prime>\alpha>1$, it follows that $\lim_{x\rightarrow(-b/a)^+}Q_{s-1}(x)=-\infty$ and $\lim_{x\rightarrow+\infty}Q_{s-1}(x)=+\infty$, hence the sign variation of $Q_{s-1}$ in $(-\frac{b}{a},+\infty)$, which is equal to the sign variation of $P_{s-1}^\prime$, is ``$-,+$'', thus the monotonicity of $P_{s-1}$ is $\searrow\nearrow$. We have $\lim_{x\rightarrow+\infty}P_{s-1}(x)=+\infty$, so, if $P_{s-1}(-\frac{b}{a})>0$, the most sign varying possibility is ``$+,-,+$'', while if $P_{s-1}(-\frac{b}{a})<0$, the most sign varying possibility is ``$-,+$''. In either case, taking into account Lemma~\ref{sign-integral}, the sign variation possibilities in $(-\frac{b}{a},+\infty)$ for $V_{s}$ are ``$+,-,+$'' or ``$-,+$'' or ``$+$''. As now, $V_{s}(-\frac{b}{a})\leq 0$, the actual sign variation for $V_{s}$ is ``$-,+$''.

\smallskip

\noindent
So, finally, the possible sign variations for $V_{s}$ as $x$ traverses from 0 to $+\infty$ are ``$+,-,+$'' or ``$-,+$'' or ``$+$'', so applying Theorem~\ref{convexity-equivalence}, the proof is concluded.
\end{proof}

\begin{proposition}
\label{Weib_Weib2}
Let $\alpha^\prime>1>\alpha>0$ and $\theta_1,\theta_2>0$. The $W(\alpha^\prime,\theta_1)$ distribution is more \sIFR{s} than the $W(\alpha,\theta_2)$ distribution.
\end{proposition}
\begin{proof}
The argument is the same as that of the proof of Proposition~\ref{gamma_gamma2}.
\end{proof}

\end{document}